\documentclass[11pt]{amsart}
\usepackage{amsmath,amsthm}
\usepackage{amssymb,latexsym}
\usepackage{enumerate}

\newtheorem{thm}{Theorem} 
\newtheorem{lem}{Lemma}

\newtheorem{cor}[thm]{Corollary}

\theoremstyle{definition}
\newtheorem{rem}{Remark}

\numberwithin{equation}{section}

\makeatletter
\@namedef{subjclassname@2010}{%
  \textup{2010} Mathematics Subject Classification}
\makeatother

\newcommand{\Leg}[2]{\left(\frac{#1}{#2}\right)}

\newcommand{\Z}{\mathbb{Z}}
\newcommand{\R}{\mathbb{R}}

\newcommand{\Oc}{\mathcal{O}}

\begin{document}


\title[On the average number of divisors]{On the average number of divisors\\ of reducible quadratic polynomials}

\author[K. Lapkova]{Kostadinka Lapkova}
\address{
Graz University of Technology\\
Institute of Analysis and Number Theory\\
Kopernikusgasse 24/II, 8010 Graz, Austria}
\email{lapkova@math.tugraz.at}

\date{17.10.2016}

\begin{abstract} We give an asymptotic formula for the divisor sum $\sum_{c<n\leq N}\tau\left((n-b)(n-c)\right)$ for integers $b<c$ of the same parity. Interestingly, the coefficient of the main term does not depend on the discriminant as long as it is a full square. We also provide effective upper bounds of the average divisor sum for some of the reducible quadratic polynomials considered before, with the same main term as in the asymptotic formula.
\end{abstract}

\subjclass[2010]{Primary 11N56; Secondary 11D09 } 
\keywords{number of divisors, quadratic polynomial, Dirichlet series}

\maketitle


\section{Introduction}\label{sec:intro}

Let $\tau(n)$ denote the number of positive divisors of the integer $n$ and $P(x)\in\Z[x]$ be a polynomial. There are many results on estimating average sums of divisors 
\begin{equation}\label{S1}\sum_{n=1}^N\tau\left(P(n)\right)\,,
\end{equation}
one of which was obtained by Erd\H os \cite{erdos}, who showed that for an irreducible polynomial $P(x)$ and for any $N>1$, we have 
$$N \log N \ll_P\sum_{n=1}^N \tau(P(n)) \ll_P N \log N\,.$$
Here the implied constants can depend both on the degree and the coefficients of the polynomial. When $P(x)$ is a quadratic polynomial Hooley \cite{hooley} and McKee \cite{mckee}, \cite{mckee2} obtained asymptotic formulae for the sum (\ref{S1}). When $\deg P(x)\geq 3$ no asymptotic formulae for (\ref{S1}) are known. A certain progress in this direction was made by Elsholtz and Tao in \S7 of \cite{elsh-tao}. 

\par  When the polynomial $P(x)$ is {\it reducible} the behavior is a little bit different. Ingham \cite{ingham} considered the additive divisor problem and proved that for a fixed positive integer $q$ the following asymptotic holds
$$\sum_{n\leq N}\tau (n)\tau(n+q)\sim \frac 6{\pi^2}\sigma_{-1}(q)N\log^2 N\,,$$
as $N\rightarrow\infty$, where $\sigma_a(q)=\sum_{d\mid q}d^a$ for $a, q\in\Z$. Later Hooley \cite{hooley} predicted that 
\begin{equation}\label{formulaHooley}\sum_{n\leq N}\tau (n^2-r^2)=A(r)N\log^2 N+\Oc (N\log N)\,,
\end{equation}
but only recently Dudek \cite{dudek} provided the exact value of the constant $A(1)$, namely $1/\zeta(2)=6/\pi^2$. The first aim of this paper is to extend Dudek's work and to find the exact values of $A(r)$ for any integer $r\geq 1$. Actually we find the main term in the asymptotic formula for \eqref{S1} for slightly more general polynomials $P(n)=(n-b)(n-c)$ for integers $b<c$, such that $b+c$ is even. 

\par For integers $k\geq 0$ and $d>0$ we define 
\begin{equation}\label{def:ro}\rho_k(d):=\#\left\{0\leq x< d\,: \quad x^2\equiv k\pmod d\right\}.
\end{equation}
The main result we need for the asymptotic estimate of the average divisor sum \eqref{S1} for reducible quadratic polynomials $P(x)$ is the following Theorem, which is of interest of its own.


\begin{thm}\label{thm1}For any integer $r\geq 1$ we have the asymptotic formula
\[\sum_{\lambda\leq N}\rho_{r^2}(\lambda)\sim \frac 6 {\pi^2}N\log N,\]
as $N\rightarrow\infty$ .
\end{thm}

From Theorem \ref{thm1} we can deduce our asymptotic result.

\begin{thm}\label{thm2}Let $b<c$ be integers with the same parity. 
Then we have the asymptotic formula
\[\sum_{c<n\leq N}\tau\left((n-b)(n-c)\right)\sim \frac 6 {\pi^2}N\log^2 N, \]
 as $N\rightarrow\infty$.
\end{thm}

Clearly the reducible quadratic polynomials of the type $n^2-r^2$ considered in \eqref{formulaHooley} are covered by Theorem \ref{thm2}, this is the case $c=-b=r>0$. It is very interesting that the constants $A(r)$ in the formula \eqref{formulaHooley} are uniform for $r\geq 1$. This would be no surprise if there is a relation of the function $\rho_{r^2}(d)$ with a certain class number, which would not change if we factor the positive discriminant with a full square. Such a correspondence was described by McKee in \cite{mckee}, \cite{mckee2}, \cite{mckee3}, however, only for square-free discriminants. 

\par Hooley \cite{hooley} suggested one possible way to get to the values $A(r)$ and prove Theorem \ref{thm2}, namely to start from Ingham's work \cite{ingham}. However we follow Dudek's method which relies on a Tauberian theorem. Let for a multiplicative function $\lambda(n)$ we denote the Dirichlet series $D_{\lambda}(s):=\sum_{n=1}^\infty \lambda(n)/n^s$. In order to find the value $A(1)$ Dudek uses information for the function $\rho_{1}(d)$ and then a Tauberian theorem for the Dirichlet series $D_{\rho_1}(s)$. It turns out that all the necessary information for $\rho_{n}(d)$ for any integer $n\geq 1$ can be extracted from section \S4 from Hooley's paper \cite{hooley1}. There Hooley investigates the Dirichlet series $D_{\rho_n}(s)$. Actually his further investigations can also lead to a proof of Theorem \ref{thm2} with an explicit error term. However, we would use these further investigations, more precisely formula (11) from \cite{hooley1}, only in the next part of the present paper where we estimate {\it explicitly} from above the average divisor sum $\sum_{c<n\leq N}\tau\left((n-b)(n-c)\right)$, much in the spirit of our earlier paper \cite{lapkova}. 

\par The motivation to consider also explicit upper bounds for the sum of divisors \eqref{S1} for quadratic polynomials comes from their application in Diophantine sets problems. Let $n\neq 0$ be an integer. A set of $m$ positive integers $\left\{a_1,\ldots , a_m\right\}$ is called a $D(n)$--$m$-tuple if $a_ia_j+n$ is a perfect square for all $i,j$ with $1\leq i<j\leq m$. The classical and most extensively studied type of such sets are the Diophantine sets $D(1)$. In our paper \cite{lapkova} we gave a similar explicit upper bound for the sum \eqref{S1} for an irreducible quadratic polynomial $P(x)$ of certain type, which allowed to improve the maximal possible number of $D(-1)$-quadruples. We believe that in a similar way the upper bounds which will be given by Theorem \ref{thm3} and Corollary \ref{cor4} stated below can be useful for estimating the number of $D(4)$, $D(16)$ or other $D(k^2)$ - sets, which are investigated in a number of papers, e.g. \cite{bliznac_filipin}, \cite{filipin1}, \cite{filipin2}, \cite{filipin3}, \cite{fujita}. 
%

\par We have the following theorem.

\begin{thm}\label{thm3}Let $b<c$ be integers with the same parity and $\delta=(b-c)^2/4$ factor as $\delta=2^t\Omega^2$ for some even $t\geq 0$ and odd $\Omega\geq 1$. Assume that $\sigma_{-1}(\Omega)\leq 4/3$. Let $c^*=\max(1,c+1)$ and $X=\sqrt{f(N)}$. Then for any integer $N\geq c^*$ we have 
\begin{align*}\sum_{c^*\leq n\leq N}\tau\left((n-b)(n-c)\right)&< 2N\left(\frac 3{\pi^2} \log^2 X+\left(\frac 6{\pi^2}+C(\Omega)\right)\log X+C(\Omega)\right)\\
&+2X\left(\frac 6{\pi^2}\log X+C(\Omega)\right)\,,
\end{align*}
where 
$$C(\Omega)=2\sum_{d\mid\Omega}\frac 1 d \left(2\sigma_0(\Omega/d)-1.749 \cdot\sigma_{-1}(\Omega/d)+1.332\right)\,.$$
 
\end{thm}

\begin{rem}When $c>0$ and close to $N$, i.e. we have relatively few summands, one can adjust the upper bound by subtracting the negative quantity $2\sum_{d\leq c}(1-c/d)\rho_\delta(d)$. In order to estimate well enough this quantity, however, we need to establish also a strong effective {\it lower} bound of the sum $\sum_{d\leq c}\rho_\delta(d)/d$, a task we do not pursue here. 
\end{rem}
The first important feature of Theorem \ref{thm3} is that under the condition 
\begin{equation}\label{cond}\sigma_{-1}(\Omega)=\sum_{d\mid\Omega}\frac 1 d\leq \frac 4 3
\end{equation}
 we can provide an explicit upper bound with the same main term as in the asymptotic formula from Theorem \ref{thm2}, because $X=N+\Oc(1)$. Second feature of Theorem \ref{thm3} is that it provides bounds for a larger family of quadratic reducible polynomials than the most studied case up to now, this for $P(n)=n^2-1$, which satisfies condition \eqref{cond}. Indeed, an immediate observation is that when $\Omega=1$ we have only one divisor $d=1$, $\sum_{d\mid\Omega}1/d=1 < 4/3$. For the case $\Omega=1$, which also includes the polynomials $P(n)=n^2-4^s$ for integer $s\geq 0$, in particular $P(n)=n^2-1$, we obtain the following corollary.

\begin{cor}\label{cor4}For any integer $N\geq 1$ we have the following claims:
\begin{enumerate}[i)]
\item Let $s$ be a nonnegative integer. Then
$$\sum_{\lambda\leq N} \frac{\rho_{4^s}(\lambda)}{\lambda}< \frac 3{\pi^2}\log^2 N+2.774\cdot \log N +2.166\,.$$
\item 
$$\sum_{n=1}^N \tau(n^2-1)<N\left(\frac 6{\pi^2}\log^2 N+5.548\cdot \log N+4.332\right)\,.$$
\end{enumerate}

\end{cor}
Previous explicit upper bounds for the average number of divisors of $P(n)=n^2-1$ were obtained by Elsholtz, Filipin and Fujita \cite{elsh} with $A(1)\leq 2$. Trudgian \cite{trud} improved this to $A(1)\leq 12/\pi^2$, Cipu \cite{cipu} got $A(1)\leq 9/\pi^2$  and very recently Cipu and Trudgian \cite{cipu-trud} also achieved the best leading coefficient $A(1)\leq 6/\pi^2$ using different method than ours. The main goal of all these papers is to bound the maximal possible number of Diophantine quintuples. In another recent paper, estimating the number of $D(4)$-quintuples, Bliznac and Filipin \cite{bliznac_filipin} showed that $A(2)\leq 6/\pi^2$.
%
%
%


\section{Asymptotic formula}

In \cite{dudek} Dudek uses the following Tauberian theorem (Theorem 2.4.1 in \cite{coj-murty}).

\begin{lem}\label{lem1} Let $\displaystyle F(s)=\sum_{n=1}^\infty\frac{a_n}{n^s}$ be a Dirichlet series with non-negative coefficients converging for $\Re(s)>1$. Suppose that $F(s)$ extends analytically at all points on $\Re(s)=1$ apart from $s=1$, and that at $s=1$ we can write
\[F(s)=\frac{H(s)}{(s-1)^{1-\alpha}}\]
for some $\alpha\in\R$ and some $H(s)$ holomorphic in the region $\Re(s)\geq 1$ and non-zero there. Then 
\[\sum_{n\leq x}a_n\sim\frac{\gamma x}{(\log x)^\alpha}\]
with 
\[\gamma:=\frac{H(1)}{\Gamma(1-\alpha)}\,,\]
where $\Gamma$ is the usual Gamma function.
\end{lem}

The key information which we need in order to extend the result of \cite{dudek} is contained in the following lemma. We write $p^\alpha||n$ when $p^\alpha\mid n$ but $p^{\alpha+1}\nmid n$. 
%

\begin{lem}\label{lem2}
Fix the integer parameter $r\geq 1$ and write simply $\rho(d)$ for the function $\rho_{r^2}(d)$. Let $\beta, t\geq 0$ be such that $p^\beta||r^2$ for $p>2$,  $2^t||r^2$ and $2\beta':=\beta$, $2t':=t$. Then for the value of the function $\rho(p^\alpha)$ at prime powers we have the following cases.
\begin{enumerate}[(i)]
\item{$p\nmid 2r$}
$$\rho(p^\alpha)=2,\text{ if   }\alpha\geq 1\,.$$
\item{$p\mid r, p\neq 2$}

$$\rho(p^\alpha)=
\begin{cases}
p^{[\frac 1 2\alpha]}&,\text{ if   } \alpha\leq\beta\,,\\
2p^{\beta'} &,\text{ if   } \alpha>\beta\,.
\end{cases}$$

\item{$p=2$}

$$\rho(2^\alpha)=
\begin{cases}
2^{[\frac 1 2\alpha]}&,\text{ if   } \alpha\leq t\,,\\
2^{t'} &,\text{ if   } \alpha=t+1\,,\\
2^{t'+1} &,\text{ if   } \alpha=t+2\,,\\
2^{t'+2} &,\text{ if   } \alpha>t+2\,.\\
\end{cases}$$

\end{enumerate}
\end{lem}

\begin{proof} The lemma follows from section \S 4 of Hooley's paper \cite{hooley1}, more precisely his cases $(a)$, $(d)$ and $(e)$. There the values of the function $\rho_n(p^\alpha)$ are examined for a general integer parameter $n$. Note that the condition that $n$ is square-free, which is imposed in the theorems of \cite{hooley1}, is not required in \S 4 \cite{hooley1}.\end{proof}

\begin{proof}[Proof of Theorem \ref{thm1}]

Let us fix the integer parameter $r\geq 1$. Consider the function
$$F(s)=\sum_{\lambda=1}^\infty\frac{\rho(\lambda)}{\lambda^s}\,.$$
Clearly $\rho(\lambda)\leq\lambda$, hence the Dirichlet series $F(s)$ is absolutely convergent for $\Re(s)>2$ and we can write it as an Euler product
\begin{equation}\label{Drho}F(s)=\prod_p \left(1+\frac{\rho(p)}{p^s}+\frac{\rho(p^2)}{p^{2s}}+\ldots\right)=:\prod_p A_p(s)\,.
\end{equation}
From the following computations it will become clear that $F(s)$ is absolutely convergent for $\Re(s)>1$.\\

According to Lemma \ref{lem2} for the factors $A_p(s)$ we obtain the following cases.
\begin{enumerate}[(i)]
\item{$p\nmid 2r$.}  Then 
\begin{equation}\label{eq:Ap1} A_p(s)=1+\frac 2{p^s}+\frac 2{p^{2s}}+\ldots=\frac{1+p^{-s}}{1-p^{-s}}\,.
\end{equation}

\item{$p\mid r, p\neq 2$.} Then we get
\begin{align}\label{eq:Ap2}A_p(s)&=1+\frac 1{p^s}+\frac p{p^{2s}}+\frac p{p^{3s}}+\ldots+\frac{p^{\beta'-1}}{p^{(\beta-2)s}}+\frac{p^{\beta'-1}}{p^{(\beta-1)s}}+\frac{p^{\beta'}}{p^{\beta s}}\left(1+\frac 2{p^s}+\frac 2{p^{2s}}+\ldots\right)\nonumber\\
&=\left(1+\frac 1{p^s}\right)\left(1+\frac{p}{p^{2s}}+\ldots+\frac{p^{\beta'-1}}{p^{(\beta-2)s}}\right)+\frac{p^{\beta'}}{p^{\beta s}}\cdot\frac{1+p^{-s}}{1-p^{-s}}\,.
\end{align}

\item{$p=2$.} In this case we have
\begin{align}\label{eq:A2}A_2(s)&=1+\frac 1{2^s}+\frac 2{2^{2s}}+\frac 2{2^{3s}}+\ldots+\frac {2^{t'}}{2^{ts}}+\frac {2^{t'}}{2^{(t+1)s}}+\frac {2^{t'+1}}{2^{(t+2)s}}+\frac {2^{t'+2}}{2^{(t+3)s}}\left(1+\frac 1{2^s}+\frac 1{2^{2s}}+\ldots\right)\nonumber\\
&=\left(1+\frac 1{2^s}\right)\left(1+\frac 2{2^{2s}}+\ldots+\frac{2^{t'}}{2^{ts}}\right)+\frac{2^{t'+1}}{2^{(t+2)s}}+\frac{2^{t'+2}}{2^{(t+3)s}}\cdot\frac{1}{1-2^{-s}}
\end{align}
\end{enumerate}

We use \eqref{eq:Ap1} and the fact that for $\Re(s)>1$
$$\frac{\zeta^2(s)}{\zeta(2s)}=\prod_p \frac{1-p^{-2s}}{(1-p^{-s})^2}=\prod_p\frac{1+p^{-s}}{1-p^{-s}}\,,$$
so we can write
$$F(s)=\frac{\zeta^2(s)}{\zeta(2s)}\cdot A_2(s)\cdot\frac{1-2^{-s}}{1+2^{-s}}\cdot\prod_{\substack{p\mid r\\p\neq 2}}A_p(s)\cdot\frac{1-p^{-s}}{1+p^{-s}}=:\frac{\zeta^2(s)}{\zeta(2s)}\cdot G(s)\,.$$

It is clear that $G(s)$ is holomorphic function in the half-plane $\Re(s)\geq 1$, though it is not obvious that it is non-zero there. By Lemma \ref{lem:nonzero} below it follows that this is indeed true, because the finitely many factors $A_p(s)$ for $p=2$ and $p\neq 2$ but $p\mid r$ are non-zero for $\Re(s)\geq 1$. Then $F(s)$ fulfills the conditions of Lemma \ref{lem1}, with $\alpha=-1$, so we obtain 
\begin{equation}\label{eq:asympR}\sum_{\lambda\leq N}\rho(\lambda)\sim \gamma N\log N
\end{equation}
with 
\begin{equation}\label{eq:c} \gamma=\lim_{s\rightarrow 1}(s-1)^2F(s)=\frac 1{\zeta(2)}G(1)\,.
\end{equation}

By \eqref{eq:Ap2} in the case $p\mid r, p\neq 2$ we have
\begin{align}\label{eq:Ap1_1}A_p(1)&=\left(1+\frac 1{p}\right)\left(1+\frac{p}{p^{2}}+\ldots+\frac{p^{\beta'-1}}{p^{\beta-2}}\right)+\frac{p^{\beta'}}{p^{\beta}}\cdot\frac{1+p^{-1}}{1-p^{-1}}\nonumber\\
&=\left(1+p^{-1}\right)\left(1+\frac{1}{p}+\frac{1}{p^2}+\ldots+\frac{1}{p^{\beta'-1}}\right)+\frac{1}{p^{\beta'}}\cdot\frac{1+p^{-1}}{1-p^{-1}}\nonumber\\
&=\left(1+p^{-1}\right)\cdot\frac{1-p^{-\beta'}}{1-p^{-1}}+\frac{1}{p^{\beta'}}\cdot\frac{1+p^{-1}}{1-p^{-1}}=\frac{1+p^{-1}}{1-p^{-1}}\left(1-p^{-\beta'}+p^{-\beta'}\right)\nonumber\\
&=\frac{1+p^{-1}}{1-p^{-1}}\,.
\end{align}

By \eqref{eq:A2} in the case $p=2$ we get
\begin{align}\label{eq:A2_1}A_2(1)&= \left(1+\frac 1{2}\right)\left(1+\frac 1{2}+\ldots+\frac{1}{2^{t'}}\right)+\frac{1}{2^{t'+1}}+\frac{1}{2^{t'+1}}\cdot\frac{1}{1-2^{-1}}  \nonumber\\
&=\left(1+2^{-1}\right)\cdot\frac{1-2^{-(t'+1)}}{1-2^{-1}}+2^{-(t'+1)}+2^{-t'}=3-3\cdot 2^{-(t'+1)}+2^{-(t'+1)}+2^{-t'}=3\,.
\end{align}

Now from \eqref{eq:Ap1_1} and \eqref{eq:A2_1} plugged in the definition of $G(s)$ it follows that $G(1)=1$. From \eqref{eq:c} it follows that $\gamma=1/\zeta(2)$ and Theorem \ref{thm1} follows from \eqref{eq:asympR}.
\end{proof}

\begin{lem}\label{lem:nonzero} If $p=2$ or $p\mid r$ the functions $A_p(s)$ have no zeros in the half-plane $\Re(s)\geq 1$. 
\end{lem}

\begin{proof}
We will verify the claim in elementary way with few cases to consider. First, we observe that from Lemma \ref{lem2} it follows that for $p\mid r$ and $p\geq 5$ we have $\rho(p^\alpha)\leq p^{\alpha/2}$ for every integer $\alpha\geq 1$. Let us write $\Re(s)=\sigma$, with $\sigma\geq 1$. Then we easily obtain
\begin{align*}
 \left|A_p(s)\right|&=\left|\sum_{\alpha=0}^\infty\frac{\rho(p^\alpha)}{p^{\alpha s}}\right|\geq 1-\left|\sum_{\alpha=1}^\infty\frac{\rho(p^\alpha)}{p^{\alpha s}}\right|\geq 1-\sum_{\alpha=1}^\infty\frac{|\rho(p^\alpha)|}{p^{\alpha \sigma}}\\
 &\geq 1-\sum_{\alpha=1}^\infty\frac{p^{\alpha/2}}{p^{\alpha \sigma}}=1-\frac{p^{-(\sigma-1/2)}}{1-p^{-(\sigma-1/2)}}>0\,.
\end{align*}

When $p=3$ and $p\mid r$, let us assume that $A_3(s)=0$. In this case 
$$A_3(s)=(1+3^s)\sum_{\alpha=0}^{\beta'-1}\frac{3^\alpha}{3^{2\alpha s}}+\frac{3^{\beta'}}{3^{\beta s}}\cdot\frac{1+3^{-s}}{1-3^{-s}}=(1+3^s)\left(\frac{1-3^{-\beta'(2s-1)}}{1-3^{-(2s-1)}}+\frac{3^{-\beta'(2s-1)}}{1-3^{-s}}\right)$$
and the expression in the second brackets should be zero, therefore
\begin{equation}\label{eq:zero3}
\frac{1-3^{-s}}{1-3^{-(2s-1)}}=\frac{3^{-\beta'(2s-1)}}{3^{-\beta'(2s-1)}-1}\,.
\end{equation}
Using triangle inequalities for the absolute values of both sides of \eqref{eq:zero3} and the simple fact $3^{-\sigma}\leq 1/3$ for $\sigma\geq 1$, we see that the absolute value of the expression on the left-hand side of \eqref{eq:zero3} is at least $1/2$, while the absolute value of the expression on the right-hand side of \eqref{eq:zero3}  is less than $1/2$ when $\beta'>1$ - a contradiction. When $\beta'=1$ \eqref{eq:zero3} gives $1=3^{-s}+3^{-(2s-1)}$ which has no solutions for $\sigma\geq 1$, again by a simple comparison of the absolute values. Thus the assumption that $A_3(s)=0$ when $\sigma\geq 1$ is wrong.

It remains to check the case $p=2$. Assume that $A_2(s)=0$ for $\sigma\geq 1$. Then from
$$A_2(s)=\frac{1+2^{-s}}{1-2^{-(2s-1)}}\left(1-2^{-(t'+1)(2s-1)}\right)+2^{-(t'+1)(2s-1)}\frac{2^s+1}{2^s-1}$$
we necessarily have
\begin{equation}\label{eq:zero2} \frac{1+2^{-s}}{1-2^{-(2s-1)}}\cdot\frac{2^s-1}{2^s+1}=\frac{2^{-(t'+1)(2s-1)}}{2^{-(t'+1)(2s-1)}-1}\,.
\end{equation}
The left-hand side of \eqref{eq:zero2} does not depend on $t'$ and we can see, again using triangle inequalities and the simple fact that $2^{-\sigma}\leq 1/2$ and $2^{-(2\sigma-1)}\leq 1/2$, that its absolute value is at least $1/9$. On the other hand, the right-hand side of \eqref{eq:zero2} equals $1/(2^{(t'+1)(2s-1)}-1)$ and its absolute value is at most $1/(2^{t'+1}-1)\leq 1/15$ for $t'\geq 3$. This gives $1/9\leq 1/15$ - a contradiction. The cases $t'\in\left\{0,1,2\right\}$ can be dealt in a similar way, substituting the corresponding value of $t'$ in \eqref{eq:zero2}, and then arranging the expressions in an equation with absolute values on the two sides which cannot be equal. Thus the assumption  $A_2(s)=0$ for $\sigma\geq 1$ is wrong and this completes the proof of the Lemma.
\end{proof}


\begin{proof}[Proof of Theorem \ref{thm2}]
Let us write $f(n):=(n-b)(n-c)$ and 
\begin{equation}\label{def:S} S(N):=\sum_{c<n\leq N}\tau\left(f(n)\right)=\sum_{c<n\leq N}\sum_{d\mid f(n)}1\,.
\end{equation}
Let $X=\sqrt{f(N)}$. Note that $f(n)$ is positive for $n>c$ and increasing. By Dirichlet hyperbola method we have
$$S(N)=\sum_{c<n\leq N}\left(2\sum_{\substack{d\leq \sqrt{f(n)}\\d\mid f(n)}}1+\Oc(1)\right)=2\sum_{c<n\leq N}\sum_{\substack{d\leq \sqrt{f(n)}\\ f(n)\equiv 0\pmod d}}1+\Oc(N)\,.$$
Recall that $\delta=(b-c)^2/4$. We notice that $f(n)=(n-b)(n-c)=\left(n-(b+c)/2\right)^2-\left((b-c)/2\right)^2=\left(n-(b+c)/2\right)^2-\delta$ and the condition $f(n)\equiv 0\pmod d$ is equivalent to $\left(n-(b+c)/2\right)^2\equiv\delta\pmod d$. 
If we denote 
\begin{equation}\label{def:M} M(x,d):=\#\left\{1\leq m\leq x \quad :\quad  f(m)\equiv 0\pmod d\right\}\,,
\end{equation}
clearly we have 
$$M(x,d)=\frac x d \rho_{\delta}(d)+\Oc\left(\rho_{\delta}(d)\right)\,.$$
Then we proceed in the standard way.
\begin{align*}\label{Sigma12}S(N)&=2\sum_{d\leq X}\sum_{\substack{1\leq n\leq N\\ f(n)\equiv 0\pmod d}}1 +\Oc\left(\sum_{d\leq X}\rho_\delta(d)\right)+\Oc(N)\nonumber\\
&=2\sum_{d\leq X}M(N,d) +\Oc\left(\sum_{d\leq X}\rho_\delta(d)\right)+\Oc(N)\nonumber\\
&= 2N\sum_{d\leq X}\frac{\rho_{\delta}(d)}{d} +\Oc\left(\sum_{d\leq X}\rho_\delta(d)\right)+\Oc(N)\,.
\end{align*}

We have that $\delta\geq 1$ is a full square of an integer, so we can apply Theorem \ref{thm1}. As $X=N+\Oc(1)$ it follows that 
$$\sum_{d\leq X}\rho_{\delta}(d)\ll X\log X\ll N\log N\,.$$
Therefore we get
\begin{equation}\label{Sigma12}
S(N)=2N\sum_{d\leq X}\frac{\rho_{\delta}(d)}{d} +\Oc\left(N\log N\right)\,.
\end{equation}
Again using Theorem \ref{thm1} and Abel's summation we get
\begin{align*} \sum_{d\leq X}\frac{\rho_{\delta}(d)}{d} &= \frac{6}{\pi^2}\int_1^X \frac{\log t}t dt+{o}\left(\int_1^X \frac{\log t}t dt\right)+\Oc\left(\frac 1 X\sum_{d\leq X}\rho_{\delta}(d)\right)\\
&=\frac 3{\pi^2}\log^2 N+o\left(\log^2 N\right)\,.
\end{align*}

The statement of Theorem \ref{thm2} follows from plugging the latter asymptotic formula into \eqref{Sigma12}.
\end{proof}

\section{Explicit upper bound}

First we will repeat the argument of Hooley \cite{hooley1} so that we recreate his formula (11) in the identity \eqref{id:Dir}. We present the details of the argument for the sake of clarity and to work out precisely the specific quantities for our special case of a square-full discriminant $\delta$. In this section we will again sometimes omit $\delta$ in the notation for the function $\rho_\delta$, since $\delta$ is fixed. Recall \eqref{Drho} where for $\Re(s)>2$ we denoted $D_{\rho}(s)=F(s)=\prod_p A_p(s)$. Also let $r^2=\delta=(b-c)^2/4$ for $f(n)=(n-b)(n-c)$ and  $\beta, t\geq 0$ be such that $p^\beta||r^2$ for $p>2$,  $2^t||r^2$ and $2\beta':=\beta$, $2t':=t$. 
\par In case $(ii)$ when $p\mid r, p\neq 2$ we can continue the expression \eqref{eq:Ap2} in the following way:
\begin{align}\label{eq:Ap2_1}A_p(s)&=(1+p^{-s})\left(1+\frac{p}{p^{2s}}+\ldots+\frac{p^{\beta '-1}}{p^{(\beta-2)s}}+\frac{p^{\beta '}}{p^{\beta s}}(1-p^{-s})^{-1}\right)\nonumber\\
&=(1+p^{-s})\sum_{\gamma^2\mid p^{\beta}}\frac{\gamma}{\gamma^{2s}}\left(1-\Leg{p^{\beta}/\gamma^2}{p}\frac{1}{p^s}\right)^{-1}\nonumber\\
&=(1+p^{-s})\sum_{\gamma^2\mid p^{\beta}}\frac{\gamma}{\gamma^{2s}}\left(1-\Leg{\delta/\gamma^2}{p}\frac{1}{p^s}\right)^{-1}
\,,
\end{align}
because $\delta=r^2$ is a full square and $p^\beta||r^2$, and $\Leg{.}{.}$ is the Jacobi symbol.
\par If $p\nmid 2r$ we notice that $\beta=0$ and the last sum over $\gamma$ from \eqref{eq:Ap2_1} equals exactly $(1-p^{-s})^{-1}$, so we can rewrite in a similar way also the factors \eqref{eq:Ap1}. For $p\nmid 2r$ we obtain
\begin{equation}\label{eq:Ap1_1} A_p(s)=(1+p^{-s})\sum_{\gamma^2\mid p^{\beta}}\frac{\gamma}{\gamma^{2s}}\left(1-\Leg{\delta/\gamma^2}{p}\frac{1}{p^s}\right)^{-1}\,.
\end{equation}
The identity \eqref{eq:A2} for case $(iii)$ when $p=2$ can be continued as
\begin{align}\label{eq:A2_1_1} A_2(s)&=(1+2^{-s})\left[\sum_{\alpha=0}^{t'}\frac{2^{\alpha}}{2^{2\alpha s}}+\frac{2^{t'+1}}{2^{(t+2)s}}(1-2^{-s})(1-2^{-2s})^{-1}+\frac{2^{t'+2}}{2^{(t+3)s}}(1-2^{-2s})^{-1}\right] \nonumber\\
&=(1+2^{-s}) K(s),
\end{align}
where $$K(s):=\sum_{\alpha=0}^\infty \frac{a_\alpha}{2^{\alpha s}}$$
and a direct calculation shows that
\begin{equation}\label{def:a_alpha}a_\alpha=
\left\{\begin{array}{ll}2^{\alpha/2}& \text{ for }0\leq\alpha\leq t \text{ and }2\mid\alpha\,,\\
0 & \text{ for }1\leq\alpha\leq t+1 \text{ and }2\nmid\alpha\,,\\
2^{t'+1} & \text{ for }\alpha\geq t+2 \,.
\end{array}\right.
\end{equation}
From \eqref{eq:Ap2_1}, \eqref{eq:Ap1_1} and \eqref{eq:A2_1_1} and $\prod_p (1+p^{-s})=\zeta(s)/\zeta(2s)$ we obtain
\begin{align*}D_\rho(s)&=K(s)\frac{\zeta(s)}{\zeta(2s)}\prod_{p=3}^\infty \sum_{\gamma^2\mid p^{\beta}}\frac{\gamma}{\gamma^{2s}}\left(1-\Leg{\delta/\gamma^2}{p}\frac{1}{p^s}\right)^{-1}\\
&=K(s)\frac{\zeta(s)}{\zeta(2s)}\sum_{\substack{d^2\mid\delta\\(d,2)=1}}\frac{d}{d^{2s}}\prod_{p=3}^\infty\left(1-\Leg{\delta/p^{2u}}{p}\frac{1}{p^s}\right)^{-1}\\
&=K(s)\frac{\zeta(s)}{\zeta(2s)}\sum_{\substack{d^2\mid\delta\\(d,2)=1}}\frac{d}{d^{2s}}\prod_{p=3}^\infty\left(1-\Leg{\delta/d^2}{p}\frac{1}{p^s}\right)^{-1}\,,
\end{align*}
where $p^u||d$ and we used that $\Leg{\delta/p^{2u}}{p}=\Leg{\delta/d^2}{p}$.
Now recall that $\delta=2^t\Omega^2$ with even $t$ and odd $\Omega$, so we can finally write
\begin{equation}\label{id:Dir}
D_\rho(s)=\sum_{\lambda=1}^\infty\frac{\rho(\lambda)}{\lambda^s}=\sum_{\alpha=0}^\infty\frac{a_\alpha}{2^{\alpha s}}\sum_{d\mid\Omega}\frac{d}{d^{2s}}\sum_{h=1}^\infty\frac{\xi_d(h)}{h^s}\,,
\end{equation}
where 
\begin{equation}\label{def:xi}\sum_{h=1}^\infty\frac{\xi_d(h)}{h^s}:=\frac{\zeta(s)}{\zeta(2s)}\prod_{p=3}^\infty\left(1-\Leg{\delta/d^2}{p}\frac{1}{p^s}\right)^{-1}=\frac{\zeta(s)}{\zeta(2s)}\sum_{\substack{l=1\\(l,2)=1}}^\infty\Leg{\delta/d^2}{l}\frac{1}{l^s}\,.
\end{equation}
We introduce the character 
\begin{equation*}\chi_d(l)=\left\{
\begin{array}{ll}\Leg{\delta/d^2}{l}, &\text{ if }2\nmid l,\\ 
0, & \text{otherwise .}\end{array}
\right.
\end{equation*}
Here $$\Leg{\delta/d^2}{l}=\Leg{2^t\Omega^2/d^2}{l}=\Leg{\Omega^2/d^2}{l}$$
is $1$ or $0$, depending on whether the condition $(l,\Omega/d)=1$ holds. This means that the character $\chi_d$ is actually the principal character modulo $2\Omega/d$, i.e.
\begin{equation}\label{def:char}\chi_d(l)=\left\{
\begin{array}{ll}
1, &\text{if   } \,\,(l,2\Omega/d)=1,\\ 
0, & \text{otherwise .}\end{array}
\right.
\end{equation}
Now we can write 
\begin{equation}\label{def:xi_d}\sum_{h=1}^\infty\frac{\xi_d(h)}{h^s}=\frac{\zeta(s)}{\zeta(2s)}\sum_{l=1}^\infty\frac{\chi_d(l)}{l^s}\,.
\end{equation}
We note that 
$$\frac{\zeta(s)}{\zeta(2s)}=\sum_{n=1}^\infty \frac{\mu^2(n)}{n^s}=D_{\mu^2}(s)$$
and from \eqref{def:xi_d} and the uniqueness of Dirichlet series expansion it follows that for every $n\geq 1$ we have
$$\xi_d(n)=\sum_{lm=n}\mu^2(l)\chi_d(m)\,.$$
By \eqref{id:Dir} valid for $\Re s>2$, again using the uniqueness of Dirichlet series expansion, it follows that for the coefficients of the corresponding Dirichlet series we have the identity
$$\sum_{\lambda\leq X}\rho(\lambda)=\sum_{d\mid\Omega}\sum_{2^\alpha d^2h\leq X}a_\alpha d\xi_d(h)\,.$$
We summarize the results up to now in the following lemma.


\begin{lem}\label{lem:Dirichlet} Let $\delta=2^{2t'}\Omega^2$ for integers $t'\geq 0$ and $\Omega\geq 1$, such that $(\Omega,2)=1$. Given the definitions \eqref{def:ro}, \eqref{def:a_alpha}, \eqref{def:xi} and \eqref{def:char}, for any $X\geq 1$ and $d\mid\Omega$, we have the identities
$$\sum_{\lambda\leq X}\rho_\delta(\lambda)=\sum_{d\mid\Omega}\sum_{2^\alpha d^2h\leq X}a_\alpha d\xi_d(h)$$
and 
$$\xi_d(n)=\sum_{lm=n}\mu^2(l)\chi_d(m)\,.$$
   \end{lem}


\subsection{Proof of Theorem \ref{thm3}} Using the notation \eqref{def:M}, and again the Dirichlet hyperbola method, we have
\begin{align*}S(N)&:=\sum_{c^*\leq n\leq N}\tau(f(n))=\sum_{c^*\leq n\leq N}\sum_{d\mid f(n)}1\leq 2\sum_{c^*\leq n\leq N}\sum_{\substack{d\leq\sqrt{f(n)}\\d\mid f(n)}}1\nonumber \\
 &\leq 2\sum_{d\leq X}\sum_{\substack{d\mid f(n)\\\max (c^*,b+d)\leq n\leq N}}1\leq 2\sum_{d\leq X}\sum_{\substack{1\leq n\leq N\\d\mid f(n)}}1=2\sum_{d\leq X}M(N,d)\,.
\end{align*}
We used that $(n-b)^2>(n-b)(n-c)\geq d^2$. One can achieve more precise upper bound for positive $c$ by more careful argument at this point. We made a cruder step by summing over all $1\leq n\leq N$ instead of considering only those $n$ which satisfy simultaneously $f(n)\geq d^2$ and $n\geq c^*$.
\par From the definition of the quantity $M(x,d)$ in \eqref{def:M} it is clear that 
\begin{equation}\label{ineq:M_rho}M(N,d)\leq \frac N d\rho(d)+\rho(d)\,.
\end{equation}
Therefore
\begin{equation}\label{eq:S_rho} S(N)\leq 2\sum_{d\leq X}\left(\frac N d\rho(d)+\rho(d)\right)=2N\sum_{d\leq X}\frac{\rho(d)}d+2\sum_{d\leq X}\rho(d)\,.
\end{equation}
In the sequel we will estimate explicitly the sum $\sum_{d\leq Y}\rho(d)$ for any $Y\geq 1$. From this we can easily extract also an explicit upper bound of  the sum $\sum_{d\leq Y}\rho(d)/d$ through Abel's summation. To achieve our goal we will use Lemma \ref{lem:Dirichlet} and the Dirichlet convolution representation it provides in a similar way as in our previous paper \cite{lapkova}, where Lemma 2.1 played a key role by providing a comfortable Dirichlet convolution. In the current case, however, we do not have factoring in familiar multiplicative functions directly of the function $\rho_\delta(d)$ for a square-free $\delta$, rather of another multiplicative function $\xi_d$ in the presentation of $\rho_\delta(d)$ for a square-full $\delta$. From one side this makes the argument more technical than in \cite{lapkova}, from the other side the character sums we consider in the present case are much simpler because we deal just with the principal character $\chi_d$.

\par From Lemma \ref{lem:Dirichlet} it follows that
\begin{equation}\label{eq:rho_xi} \sum_{\lambda\leq X}\rho(\lambda)=\sum_{d\mid\Omega}\sum_{2^\alpha d^2\leq X}a_\alpha d\sum_{h\leq X/(2^\alpha d^2)}\xi_d(h)\,.
\end{equation}
We concentrate first on estimating the innermost sum.


\subsubsection{Estimation of $\displaystyle\sum_{h\leq Y}\xi_d(h)$} Let $Y\geq 1$ and fix a divisor $d$ of $\Omega$. By the Dirichlet convolution representation in Lemma \ref{lem:Dirichlet} it follows that
\begin{equation}\label{eq:xi_mu_chi}\sum_{h\leq Y}\xi_d(h)=\sum_{lm\leq Y}\mu^2(l)\chi_d(m)=\sum_{l\leq Y}\mu^2(l)\sum_{m\leq Y/l}\chi_d(m)\,.\end{equation}
By \eqref{def:char} for a real $Y\geq 1$, and writing temporarily $q:=2\Omega/d$ for the conductor of the principal character $\chi_d$, we have
\begin{align}\label{eq:chi_Y} \sum_{m\leq Y}\chi_d(m) &=\sum_{m\leq Y}1-\sum_{\substack{m\leq Y\\(m,q)>1}}1=\sum_{m\leq Y}1-\sum_{\substack{f\mid q\\f\geq 2}}\sum_{m_1 \leq Y/f}1\nonumber\\
&=\left[Y\right]-\sum_{\substack{f\mid q\\f\geq 2}}\left[\frac Y f\right]\leq Y-\sum_{\substack{f\mid q\\f\geq 2}}\left(\frac Y f -1\right)\,.
\end{align}
Let us use the notation
$$\theta_a(q):=\sum_{\substack{d\mid q\\d\geq 2}}d^a=\sigma_a(q)-1\,.$$ 
Then the inequality \eqref{eq:chi_Y} can be written as
\begin{equation}\label{eq:chi_sigma} \sum_{m\leq Y}\chi_d(m) \leq Y\left(1-\theta_{-1}(q)\right)+\theta_0(q)\,.
\end{equation}
Plugging this in \eqref{eq:xi_mu_chi} we get
\begin{align}\label{eq:xi_Y}\sum_{h\leq Y}\xi_d(h)&\leq\sum_{l\leq Y}\mu^2(l)\left(\frac Y l \left(1-\theta_{-1}(q)\right)+\theta_0(q) \right)\nonumber\\
&=Y\left(1-\theta_{-1}(q)\right)\sum_{l\leq Y}\frac{\mu^2(l)}l+\theta_0(q)\sum_{l\leq Y}\mu^2(l)\,.
\end{align}
At this step we need to have $1-\theta_{-1}(q)=1-\theta_{-1}(2\Omega/d)\geq 0$ for every divisor $d\mid\Omega$. First we see that if $k\geq 1$ is odd, then 
\begin{equation}\label{eq:theta_sigma}\theta_{-1}(2k)=\frac 3 2 \sigma_{-1}(k)-1\,.
\end{equation}
Indeed, we notice that all divisors of $2k$ which are at least $2$ can be presented as $f=2^\gamma g$ for $\gamma\in\left\{0,1\right\}$ and $g\mid k$, with $\gamma=1$ when $g=1$, because $k$ is odd. Thus
$$\sum_{\substack{f\mid 2k\\f\geq 2}}\frac 1 f =\sum_{g\mid k}\frac 1 g - 1 +\frac 1 2 \sum_{g\mid k}\frac 1 g =\frac 3 2 \sum_{g\mid k}\frac 1 g -1\,.$$
Then for every $d\mid\Omega$ we indeed have
$$1-\theta_{-1}(2\Omega/d)=2-\frac 3 2\sigma_{-1}(\Omega/d)\geq 2-\frac 3 2\sigma_{-1}(\Omega)\geq 0\,,$$
because we have assumed the condition \eqref{cond}.

\par Now we can use an upper bound due to Ramar\'e (Lemma 3.4 \cite{ramare}).
\begin{lem}(Ramar\'e, \cite{ramare})\label{lemRamare} Let $x\geq 1$ be a real number. We have
$$\sum_{n\leq x}\frac{\mu^2(n)}n\leq \frac 6{\pi^2}\log x+1.166\,.$$
\end{lem}
Using also the trivial bound for the last sum in \eqref{eq:xi_Y}, we get
\begin{align}\sum_{h\leq Y}\xi_d(h)&\leq Y\left(1-\theta_{-1}(q)\right)\left(\frac 6{\pi^2}\log Y+1.166\right)+\theta_0(q)Y\nonumber\\
&=\frac 6{\pi^2}\left(1-\theta_{-1}(q)\right)Y\log Y+\left(1.166\left(1-\theta_{-1}(q)\right)+\theta_0(q)\right)Y\nonumber\\
&=:c_1(q)Y\log Y+c_2(q)Y\,.
\end{align}


\subsubsection{Estimation of $\displaystyle\sum_{\lambda\leq X}\rho_\delta(\lambda)$}

Now we plug the latter bound in the first identity from Lemma \ref{lem:Dirichlet}. Using that $c_1(q)\geq 0$, we get
\begin{align*}
\sum_{\lambda\leq X}\rho_\delta(\lambda)&=\sum_{d\mid\Omega}\sum_{2^\alpha d^2\leq X}a_\alpha d\sum_{h\leq X/(2^\alpha d^2)}\xi_d(h)\nonumber\\
&\leq \sum_{d\mid\Omega}\sum_{2^\alpha d^2\leq X}a_\alpha d\left(c_1(q)\frac X{2^\alpha d^2}\log{\frac X{2^\alpha d^2}}+c_2(q)\frac X{2^\alpha d^2}\right)\nonumber\\
&\leq \frac 6{\pi^2}X\log X\sum_{d\mid\Omega}\frac 1 d\left(1-\theta_{-1}(2\Omega/d)\right)\sum_{2^\alpha\leq X/d^2}\frac{a_\alpha}{2^\alpha}+X\sum_{d\mid\Omega}\frac 1 dc_2(2\Omega/d)\sum_{2^\alpha\leq X/d^2}\frac{a_\alpha}{2^\alpha}\,.
\end{align*}

By a direct calculation using \eqref{def:a_alpha}, or by \eqref{eq:A2_1} and \eqref{eq:A2_1_1}, we see that 
\begin{equation}\label{eq:K1}
K(1)=\sum_{\alpha=0}^\infty\frac{a_\alpha}{2^\alpha}=2\,.
\end{equation}

Therefore we can bound the partial sums of $a_\alpha/2^\alpha$ by $2$ and we obtain 
\begin{equation}\label{rho_C_d} 
\sum_{\lambda\leq X}\rho_\delta(\lambda)\leq  \frac {12}{\pi^2}C_1(\Omega)X\log X+C_2(\Omega)X\,,
\end{equation}
where 
\begin{equation}\label{def:C1_Omega}
C_1(\Omega):=\sum_{d\mid\Omega}\frac 1 d\left(1-\theta_{-1}(2\Omega/d)\right)=\sum_{d\mid\Omega}\frac 1 d\left(1-\sum_{\substack{f\mid 2\Omega/d\\f\geq 2}}\frac 1 f\right)
\end{equation}
and $$C_2(\Omega):=2\sum_{d\mid\Omega}\frac 1 dc_2(2\Omega/d)=2\sum_{d\mid\Omega}\frac 1 d\left(1.166\left(1-\sum_{\substack{f\mid 2\Omega/d\\f\geq 2}}\frac 1 f\right)+\sum_{\substack{f\mid 2\Omega/d\\f\geq 2}}1\right)\,.$$

For the constant $C_1(\Omega)$ we have the following crucial upper bound which would guarantee the right main term.

\begin{lem}\label{lem:C1} Let $\Omega\geq 1$ be an odd integer satisfying $\sigma_{-1}(\Omega)\leq 4/3$. Then the constant $C_1(\Omega)$ defined in \eqref{def:C1_Omega} satisfies the inequalities
$$0<C_1(\Omega)\leq \frac 1 2\,.$$
\end{lem}

\begin{proof} 
From \eqref{eq:theta_sigma} it follows that
$$\theta_{-1}(2\Omega/d)=\frac 3 2\sigma_{-1}(\Omega/d)-1\,.$$
Then 
$$C_1(\Omega)=\sum_{d\mid\Omega}\frac 1 d \left(2-\frac 3 2 \sum_{g\mid \Omega/d}\frac 1 g \right)=2\sum_{d\mid\Omega}\frac 1 d -\frac 3 2 \sum_{d\mid\Omega}\frac 1 d \sum_{g\mid \Omega/d}\frac 1 g\,.
$$
When we take out the contribution of $d=1$ from the first sum and of $d=1$ and $g=1$ from the second sum we get
\begin{align*}
C_1(\Omega)&=2+2\sum_{\substack{d\mid\Omega\\d>1}}\frac 1 d -\frac 3 2 -\frac 3 2\sum_{\substack{g\mid\Omega\\g>1}}\frac 1 g -\frac 3 2\sum_{\substack{d\mid\Omega\\d>1}}\frac 1 d\sum_{g\mid\Omega/d}\frac 1 g\\
&=\frac 1 2+\frac 1 2\sum_{\substack{d\mid\Omega\\d>1}}\frac 1 d -\frac 3 2\sum_{\substack{d\mid\Omega\\d>1}}\frac 1 d\sum_{g\mid\Omega/d}\frac 1 g\,.
\end{align*}
If there are divisors $d\mid\Omega$ which are greater than $1$, then the innermost sums satisfy $\sum_{g\mid\Omega/d}1/g\geq 1$ with contribution at least from $g=1$. Then 
$$C_1(\Omega)\leq \frac 1 2 +\frac 1 2\sum_{\substack{d\mid\Omega\\d>1}}\frac 1 d -\frac 3 2\sum_{\substack{d\mid\Omega\\d>1}}\frac 1 d=\frac 1 2 -\sum_{\substack{d\mid\Omega\\d>1}}\frac 1 d \leq \frac 1 2\,.$$
In particular, when $\Omega=1$, we have $C_1(\Omega)=1/2$. When $\Omega>1$, by the condition \eqref{cond} we know that for every $d\mid\Omega$ we have $1-\theta_{-1}(2\Omega/d)\geq 0$. Note that in this case there is at least one prime divisor $p\geq 3$ such that $p\mid\Omega$ and we have the strict inequality $0\leq 1-\theta_{-1}(2\Omega)<1-\theta_{-1}(2\Omega/p)$. Then by \eqref{def:C1_Omega} we surely have $C_1(\Omega)>0$.
\end{proof}

From Lemma \ref{lem:C1} and \eqref{rho_C_d} we arrive at
$$\sum_{\lambda\leq X}\rho_\delta(\lambda)\leq  \frac {6}{\pi^2}X\log X+C_2(\Omega)X\,.$$
From \eqref{eq:theta_sigma} and the analogous observation $\theta_0(2\Omega/d)=\sum_{f\mid 2\Omega/d, f\geq 2}1=2\sigma_0(\Omega/d)-1$ we check by a direct calculation that $C_2(\Omega)=C(\Omega)$. Therefore 
\begin{equation}\label{rho_bound}
\sum_{\lambda\leq X}\rho_\delta(\lambda)\leq  \frac {6}{\pi^2}X\log X+C(\Omega)X\,.
\end{equation}
 

\subsubsection{Estimation of $\displaystyle\sum_{\lambda\leq X}\rho_\delta(\lambda)/\lambda$} Write $P(X):=\sum_{\lambda\leq X}\rho_\delta(\lambda)$. By \eqref{rho_bound} and Abel's summation formula we have
\begin{align}\label{rho_lambda_bound} \sum_{\lambda\leq X}\frac{\rho_\delta(\lambda)}{\lambda}&=\frac{P(X)}{X}  -\int_1^X P(u)\left(\frac 1 u\right)'du=\frac{P(X)}{X}  +\int_1^X P(u)\frac {du} {u^2}\nonumber\\
&\leq \frac 1 X\left(\frac {6}{\pi^2}X\log X+C(\Omega)X\right)+\int_1^X \left(\frac {6}{\pi^2}u\log u+C(\Omega)u\right)\frac {du} {u^2}\nonumber\\
&\leq \frac {6}{\pi^2}\log X+C(\Omega)+\frac {6}{\pi^2}\int_1^X \frac{\log u}{u}du+C(\Omega)\int_1^X\frac{du}{u}\nonumber\\
&=\frac {3}{\pi^2}\log^2 X+\left(\frac {6}{\pi^2}+C(\Omega)\right)\log X+C(\Omega)\,.
\end{align}

Now using \eqref{rho_bound} and \eqref{rho_lambda_bound}, the inequality \eqref{eq:S_rho} turns into
$$S(N)\leq 2N\left[\frac {3}{\pi^2}\log^2 X+\left(\frac {6}{\pi^2}+C(\Omega)\right)\log X+C(\Omega)\right]+2\left[\frac {6}{\pi^2}X\log X+C(\Omega)X\right]\,,$$
where $X=\sqrt{f(N)}$. This proves Theorem \ref{thm3}.

\subsection{Proof of Corollary \ref{cor4}} Instead of directly applying Theorem \ref{thm3} we will take use of the specific form of $\delta$ and $f(n)$. This way we can gain better minor terms coefficients, whereas the main coefficient remains the right one. When $\delta=4^s$, i.e. $\Omega=1$, then clearly we have only one divisor of $\Omega$ and a single character to consider : $\chi_d=\chi_1$ which is $1$ at even numbers and $0$ at odd. Then 
$$\sum_{m\leq Y}\chi_1(m)=\left[\frac{Y+1}2\right]\leq \frac{Y+1}2\,.$$
From \eqref{eq:xi_mu_chi} and Lemma \ref{lemRamare} we get
\begin{align*}\sum_{h\leq Y}\xi_1(h)&=\sum_{l\leq Y}\mu^2(l)\sum_{m\leq Y/l}\chi_1(l)\leq \sum_{l\leq Y}\mu^2(l)\left(\frac Y{2l}+\frac 1 2\right)\\
&=\frac Y 2\sum_{l\leq Y}\frac{\mu^2(l)}l+\frac 1 2 \sum_{l\leq Y}\mu^2(l)\leq\frac Y 2\left(\frac 6{\pi^2}\log Y +1.166\right)+\frac Y 2\\
&=\frac 3{\pi^2}Y\log Y +2.166\frac Y 2\,.
\end{align*}
Then by Lemma \ref{lem:Dirichlet}, the latter inequality and \eqref{eq:K1} we see that
\begin{align*}\sum_{\lambda\leq X}\rho_{4^s}(\lambda)&=\sum_{2^\alpha\leq X}a_\alpha\sum_{h\leq X/2^\alpha}\xi_1(h)\leq \sum_{2^\alpha\leq X}a_\alpha\left(\frac 3{\pi^2}\frac{X}{2^\alpha}\log \frac{X}{2^\alpha} +2.166\frac{X}{2\cdot 2^\alpha}\right) \\
&\leq \frac 3{\pi^2} X\log X  \sum_{2^\alpha\leq X}\frac{a_\alpha}{2^\alpha}+ \frac {2.166}{2} X\sum_{2^\alpha\leq X}\frac{a_\alpha}{2^\alpha}<\frac 6{\pi^2} X\log X+2.166\cdot X
\end{align*}
From the last inequality and applying Abel's summation we obtain
$$\sum_{\lambda\leq N} \frac{\rho_{4^s}(\lambda)}{\lambda}\leq \frac 3{\pi^2}(\log N)^2+\left(\frac 6{\pi^2}+2.166\right) \log N +2.166\,.$$
and the statement of $i)$ follows after a decimal approximation of the second coefficient.\\

For the proof of $ii)$ we note that
$$\sum_{n=2}^N\tau(n^2-1)\leq 2\sum_{n=2}^N \sum_{\substack{1\leq d<n\\d\mid (n^2-1)}}1=2\sum_{d=1}^N\sum_{\substack{d<n\leq N\\n^2\equiv 1(d)}}1=2\sum_{d=1}^N(M(N,d)-M(d,d))\,.$$
Clearly $M(d,d)=\rho_1(d)$ and by \eqref{ineq:M_rho} it follows that
$$\sum_{n=2}^N\tau(n^2-1)\leq 2N\sum_{d=1}^N\frac{\rho_1(d)}{d}\,.$$
The second statement of Corollary \ref{cor4} follows from applying the inequality $i)$ to the innermost sum. 

\subsection{Examples} To provide an explicit upper bound for the average divisor sum over any reducible quadratic polynomial $f(n)=(n-b)(n-c)$ with $\delta=2^{t}\Omega^2$ where $\Omega$ is odd, $t\geq 0$ is even, Theorem \ref{thm3} requires the condition \eqref{cond} for the divisors of $\Omega$. In Corollary \ref{cor4} we showed an improved such upper bound, which is valid when $\Omega=1$ and $f(n)=n^2-1$. In this subsection we would give some more examples when Theorem \ref{thm3} holds. 
\begin{enumerate}[I.]
\item \underline{$\Omega=p, p\geq 3$ is a prime.}
Indeed, then 
$$\sigma_{-1}(\Omega)=\sum_{d\mid\Omega}\frac 1 d=1+\frac 1 p \leq 1+\frac 1 3=\frac 4 3\,.$$
\item \underline{$\Omega=p^k, p\geq 5$ for integer $k\geq 2$.} Indeed, if $p=3$ and $k\geq 2$ condition \eqref{cond} fails. If $p\geq 5$ we need to have
$$\sum_{d\mid\Omega}\frac 1 d=1+\frac 1 p + \ldots +\frac 1 {p^k}=\frac{1-p^{-(k+1)}}{1-p^{-1}}\leq \frac 4 3\,,$$ which is equivalent to $1-1/p^{k+1}\leq 4/3 -4/(3p)$, or further to $4/(3p)-1/p^{k+1}\leq 1/3$. The latter is true because $4/(3p)-1/p^{k+1}<4/(3p)\leq 4/(3\cdot 5)<1/3$.\\

\item \underline{$\Omega=pq, 5\leq p<q$ and $q\geq 3(p+1)/(p-3)$.} Like in example II. one sees that if $p=3$ the sum $\sigma_{-1}(\Omega)>4/3$. Then one easily obtains the second condition on $q$ starting from the necessary inequality
$$\sum_{d\mid\Omega}\frac 1 d=1+\frac 1 p + \frac 1 q+\frac 1 {pq}\leq \frac 4 3\,.$$
Thus when $p=5$ we can have any $q\geq 11$.

\end{enumerate}

\subsection*{Funding.} 
This work was supported by the  Austrian Science Fund (FWF) [ T846-N35]; and partially by the National
Research, Development and Innovation Office (NKFIH) [K104183].



\begin{thebibliography}{HD}


\normalsize
\baselineskip=17pt

\bibitem{bliznac_filipin} M. Bliznac, A. Filipin, Upper bound for the number of $D(4)$-quintuples, \emph{Bull. Aust. Math. Soc.}, to appear
\bibitem{coj-murty} A. C. Cojocaru, M. R. Murty, \emph{An introduction to sieve methods and their applications}, Cambridge University Press, 2006
\bibitem{cipu} M. Cipu, Further remarks on Diophantine quintuples,\emph{Acta Arith.} 168 (2015), 201--219.
\bibitem{cipu-trud} M. Cipu, T. Trudgian, Searching for Diophantine quintuples, \emph{Acta Arith.}, published online 18 May 2016
\bibitem{dudek} A. Dudek, On the number of divisors of $n^2-1$, \emph{Bull. Aust. Math. Soc.} 93 (2016), no. 02, 194--198
\bibitem{elsh} C. Elsholtz, A. Filipin, Y. Fujita, On Diophantine quintuples and $D(-1)$-quadruples, \emph{Monatsh. Math.}  175  (2014),  no. 2, 227--239
\bibitem{elsh-tao} C. Elsholtz, T. Tao, Counting the number of solutions to the Erd\H os-Straus equation on unit fractions, \emph{ J. Aust. Math. Soc.}  94  (2013),  no. 1, 50--105
\bibitem{erdos} P. Erd\H os, On the sum $\sum_{k=1}^x d(f(k))$, \emph{J. London Math. Soc.} 27 (1952), 7--15
\bibitem{filipin1}A. Filipin, Extensions of some parametric families of D(16)-triples, \emph{Int. J. Math. Math. Sci.} (2007), Art. ID 63739, 12 pp.

\bibitem{filipin2}A. Filipin, An irregular D(4)-quadruple cannot be extended to a quintuple, \emph{Acta Arith.}  136  (2009),  no. 2, 167--176

\bibitem{filipin3}A. Filipin, There are only finitely many D(4)-quintuples, \emph{Rocky Mountain J. Math.} 41  (2011),  no. 6, 1847--1859.
\bibitem{fujita} Y. Fujita, Extensions of the $D(\pm k^2)$-triples ${k^2,k^2\pm 1,4k^2\pm 1}$, \emph{Period. Math. Hungar.}  59  (2009),  no. 1, 81--98

\bibitem{hooley1} C. Hooley, On the representation of a number as the sum of a square and a product, \emph{Math. Z.}  69 (1958), 211--227
\bibitem{hooley} C. Hooley, On the number of divisors of quadratic polynomials, \emph{Acta Math.} 110 (1963), 97--114
\bibitem{ingham} A. E. Ingham, Some asymptotic formulae in the theory of numbers, \emph{J. London Math. Soc.} 2 (1927), 202--208
\bibitem{lapkova} K. Lapkova, Explicit upper bound for an average number of divisors of quadratic polynomials, \emph{Arch. Math. (Basel)} 106 (2016), no. 3, 247--256
\bibitem{mckee} J. McKee, On the average number of divisors of quadratic polynomials, \emph{Math. Proc. Camb. Philos. Soc.} 117 (1995), 389--392 
\bibitem{mckee2} J. McKee, A note on the number of divisors of quadratic polynomials. Sieve methods, exponential sums, and their applications in number
 theory (Cardiff, 1995), 275--281, \emph{London Math. Soc. Lecture Note Ser.} 237, Cambridge Univ. Press, Cambridge,  1997 
\bibitem{mckee3} J. McKee, The average number of divisors of an irreducible quadratic polynomial, \emph{Math. Proc. Cambridge Philos. Soc.}  126  (1999),  no. 1, 17--22 
\bibitem{ramare} O. Ramar\'e, \textit{An explicit density estimate for Dirichlet L-series}, Math. Comp. 85 (2016), no. 297, 325--356
\bibitem{trud} T. Trudgian, \textit{Bounds on the number of Diophantine quintuples}, J. Number Theory 157 (2015), 233--249

\end{thebibliography}
\end{document}